\documentclass[a4paper,10pt]{article}
\usepackage{mathtext}
\usepackage[T1,T2A]{fontenc}
\usepackage[cp1251]{inputenc}
\usepackage[english]{babel}
\usepackage{amsmath}
\usepackage{amsfonts}
\usepackage{amssymb}
\usepackage{mathrsfs}
\usepackage{amsthm}
\usepackage{enumerate}
\usepackage{graphicx}
\graphicspath{}
\DeclareGraphicsExtensions{.pdf,.png,.jpg}

\usepackage{color}
\usepackage{euscript}

\usepackage{cite}

\newtheorem{Le}{Lemma}[section]
\newtheorem{Def}{Definition}[section]
\newtheorem{St}[Le]{Proposition}
\newtheorem{Th}{Theorem}[section]
\newtheorem{Cor}[Le]{Corollary}
\newtheorem{Rem}[Le]{Remark}

\numberwithin{equation}{section}

\DeclareMathOperator{\supp}{supp}

\DeclareMathOperator{\diam}{diam}

\title{A measure with small support and $p$-summable Fourier transform.}
\author{Nikita Dobronravov\footnote{Supported by Theoretical Physics and Mathematics Advancement Foundation "BASIS" grant Junior Leader (Math) 21-7-2-12-2 and  by the Ministry of Science and Higher Education of the Russian Federation (agreement no. 075-15-2022-287).}}
\begin{document}
	\maketitle
	\begin{abstract}
		 We construct a probability measure $\mu$ supported on a set of zero $2d/p$\! -Hausdorff measure such that $\hat{\mu}\in L_{p}(\mathbb{R}^d)$.
	\end{abstract}
	
	\section{Introducion}
	The Uncertainty Principle (UP) in mathematical analysis is a family of facts that state: both function and its Fourier transform cannot be simultaneously small (see~\cite{HavJor}). The following theorem is yet another manifestation of the Uncertainty Principle. Denote the $\alpha$-Hausdorff measure by $\mathcal{H}_\alpha$.
	\begin{Th}\label{nonend} 
		Let $S\subset \mathbb{R}^d$ be a compact set such that $\mathcal{H}_{\alpha}(S)<\infty$. Let $\zeta$ be a tempered distribution such that $\supp(\zeta)\subset S$ and $\hat{\zeta}\in L_p(\mathbb{R}^d)$ for $p<\frac{2d}{\alpha}$. Then $\zeta=0$.
	\end{Th}

	The case of Theorem~\ref{nonend} $d=1$ and $\zeta$ being a measure was considered by Salem. He also proved that for $d=1$ Theorem~\ref{nonend} cannot be strengthened to any $p>\frac{2}{\alpha}$ (see~\cite{Salem}).  
	Beurling obtained a result  that implies Theorem~\ref{nonend} in the case $d=1$ (see~\cite{Beurling}).
	Edgar and Rosenblatt  proved Theorem~\ref{nonend} in the case $d-1\leqslant \alpha$ (see~\cite{EdgRos}). Kahane obtained the following result (see~\cite{Kah}). 
	
	\begin{Th}[Kahane \cite{Kah}]\label{Kahan}
		Let $S\subset \mathbb{R}^d$ be a compact set such that $\mathcal{H}_{\alpha}(S)<\infty$. Let $\zeta$ be a distribution such that $\supp(\zeta)\subset S$ and $\zeta\in  W_2^{\frac{\alpha-d}{2}}(\mathbb{R}^d)$. Then $\zeta=0$.
	\end{Th}
	Here $W^{\alpha}_p(\mathbb{R}^d)$ is the potential Sobolev space,  defined for $p\in(1,\infty)$ by the following formula: 
	
	\begin{equation}
		W^{\alpha}_p(\mathbb{R}^d)=\{f| \  (1-\Delta)^{\frac{\alpha}{2}}f\in L_p(\mathbb{R}^d)\}.
	\end{equation}
	We use the notation $(1-\Delta)^{\frac{\alpha}{2}}$ for 
	\begin{equation}
		((1-\Delta)^{t}f)^{\widehat{}}(\xi)=|\hat{f}(\xi)|^2(1+|2\pi \xi|^2)^{t}.
	\end{equation}
	Theorem~\ref{nonend} can be deduced from Theorem~\ref{Kahan} because the Fourier transform maps $L_p(\mathbb{R}^d)$ to $W_2^{\frac{d}{p}-\frac{d}{2}-\varepsilon}(\mathbb{R}^d)$ for $p>2$ and $\varepsilon>0$:
	\begin{multline}
		\|f\|^2_{W_2^{\frac{d}{p}-\frac{d}{2}-\varepsilon}(\mathbb{R}^d)}=
		\int_{\mathbb{R}^d}|\hat{f}(\xi)|^2(1+|2\pi \xi|^2)^{\frac{d}{p}-\frac{d}{2}-\varepsilon}d\xi\leqslant\\
		\left(\int_{\mathbb{R}^d}|\hat{f}(\xi)|^pd\xi\right)^{\frac{2}{p}}\left(\int_{\mathbb{R}^d}(1+|2\pi \xi|^2)^{-\frac{d}{2}-\frac{\varepsilon p}{p-2}}d\xi\right)^{\frac{p-2}{p}}\lesssim\|f\|_{L_p}^2.
	\end{multline}
	Here and in what follows $A\lesssim B$ means there exists $C$ such that $A\leqslant CB$ and $C$ is uniform in certain sense.

	Adams and Polking proved a result that implies Theorem~\ref{nonend} (see~\cite[Theorem A]{AdPo}). However, they did not formulate anything resembling Theorem~\ref{nonend}.
	
	The limit case ($\alpha=2d/p$) was still open in the generality of Theorem~\ref{nonend}: it was unknown whether there exists an non zero measure $\mu$ supported on a set $S$, $\mathcal{H}_\alpha(S)<\infty$, and such that $\hat{\mu}\in L_\frac{2d}{\alpha}(\mathbb{R}^d)$. If one makes additional regularity assumptions on $S$, the UP holds true.
	
	
	\begin{Th}[Rosenblatt \cite{Ros}]\label{Rosenblatt}
		Let $S\subset \mathbb{R}^d$ be a $(d-1)$ dimensional smooth surface. Let $\zeta$ be a distribution such that $\supp(\zeta)\subset S$ and $\hat{\zeta}\in  L_{\frac{2d}{d-1}}(\mathbb{R}^d)$. Then $\zeta=0$.
	\end{Th}

	\begin{Th}[Agranovskiy, Narayanan \cite{ArgNar}]\label{AN}
		Let $S\subset \mathbb{R}^d$ be a $C^1$ surface of dimension $k$. Let $\zeta$ be a distribution such that $\supp(\zeta)\subset S$ and $\hat{\zeta}\in  L_{\frac{2d}{k}}(\mathbb{R}^d)$. Then $\zeta=0$.
	\end{Th}

	\begin{Th}[Raani \cite{Raani}]\label{Raani} 
		Let $S\subset \mathbb{R}^d$ be a compact set such that $\mathcal{P}_{\alpha}(S)<\infty$. Let $\zeta$ be a distribution such that $\supp(\zeta)\subset S$ and $\hat{\zeta}\in L_{\frac{2d}{\alpha}}(\mathbb{R}^d)$. Then $\zeta=0$.
	\end{Th}
	Here $\mathcal{P}_{\alpha}$ is the $\alpha$-packing measure (see~\cite[Chapter~5]{Mattila} for more information about packing measures). Theorem~\ref{Raani} is a generalization of Theorems~\ref{Rosenblatt} and~\ref{AN}.	
	Observe $\mathcal{P}_{\alpha}(\cdot)\geqslant\mathcal{H}_\alpha(\cdot)$. This inequality says that packing measure "sees" smaller sets than the Hausdorff measure. 
	
		The main result of the paper shows this UP does not hold at the endpoint without additional structural assumptions.
	\begin{Th}\label{res}
		Let $2<p<\infty$. There exists a compact set $S\subset \mathbb{R}^d$ and a probability measure $\mu$ such that $\supp(\mu)\subset S$, $\hat{\mu}\in L_{p}$, and $\mathcal{H}_{\frac{2d}{p}}(S)=0$.
	\end{Th}
	In particular, one cannot generalise Theorem~\ref{Raani} to Hausdorff measures.	The note is devoted to the proof of Theorem~\ref{res}, which is an explicit construction of a specific random Cantor-type set.
	Other constructions of random Cantor sets were used by Salem (see~\cite{Salem}) and Bluhm (see~\cite{Bluhm}). Our construction is more complicated and the Cantor set develops differently along different branches of the corresponding tree. This allows to distinguish the Lorentz spaces. Though these spaces play the pivotal role, we prefer to omit them in our notation. 
	Considerations from which we construct $\mu$ and $S$ in Theorem~\ref{res} are also related with Netrusov--Hausdorff capacities. This paper does not use them as well to make things simpler. An article with the exact form of the Uncertainty Principle for Lorentz spases and  Netrusov--Hausdorff capacities will appear elsewhere.
	
 	In Section~\ref{Constr}, we give a construction of $S$ and $\mu$. In Section~\ref{ESC}, we prove that our construction satisfies the conditions of Theorem~\ref{res}. In Section~\ref{AL}, we prove some auxiliary lemmas.

	{\bf Acknowledgment.}
	I am grateful to my scientific adviser D. M. Stolyarov for statement of the problem and attention to my work. I also wish to thank  M. K. Dospolova and A. S. Tselishchev for reading my work and advice on article formatting. 
	
	\section{Construction}\label{Constr}
	
	\subsection{General construction}
	Let $\mathcal{M}={M_0,M_1,...}$ be an infinite sequence of natural numbers to be specified later.
	We start with the construction of an infinite tree $\mathcal{T}$. We will inductively construct its subtrees $\mathcal{T}_k$ such that $\mathcal{T}_k\subset\mathcal{T}_{k+1}$ and $\mathcal{T}=\cup\mathcal{T}_k$.
	
	Let $V(G)$ be the set of vertices of the graph $G$ and let $E(G)$ be the set of edges. Set $V(\mathcal{T}_0)=\{Q_0,Q_1,\dots,Q_{M_0}\}$ and $E(\mathcal{T}_0)=\{(Q_0,Q_1),(Q_0,Q_2),\dots,(Q_0,Q_{M_0})\}$. 
	
	Assume we have concrusted $\mathcal{T}_{k-1}$. To build $\mathcal{T}_k$ we add to $\mathcal{T}_{k-1}$  $M_k$ new vertices connect them with $Q_k$:
	\begin{equation}
		V(\mathcal{T}_k)=V(\mathcal{T}_{k-1})\cup\{Q_{M_0+\dots+M_{k-1}+1},Q_{M_0+\dots+M_{k-1}+2},\dots ,Q_{M_0+\dots+M_{k-1}+M_k}\},
	\end{equation}
	\begin{equation}
		E(\mathcal{T}_k)=E(\mathcal{T}_{k-1})\cup\{(Q_k,Q_{M_0+\dots+M_{k-1}+1}),(Q_k,Q_{M_0+\dots+M_{k-1}+2}),\dots ,(Q_k,Q_{M_0+\dots+M_{k-1}+M_k})\}.
	\end{equation}
	We will say that $Q_k$ is a paren of those new vertices and that they are kids of $Q_k$.
	
	\begin{figure}[h]
		\center{\includegraphics[scale=0.5]{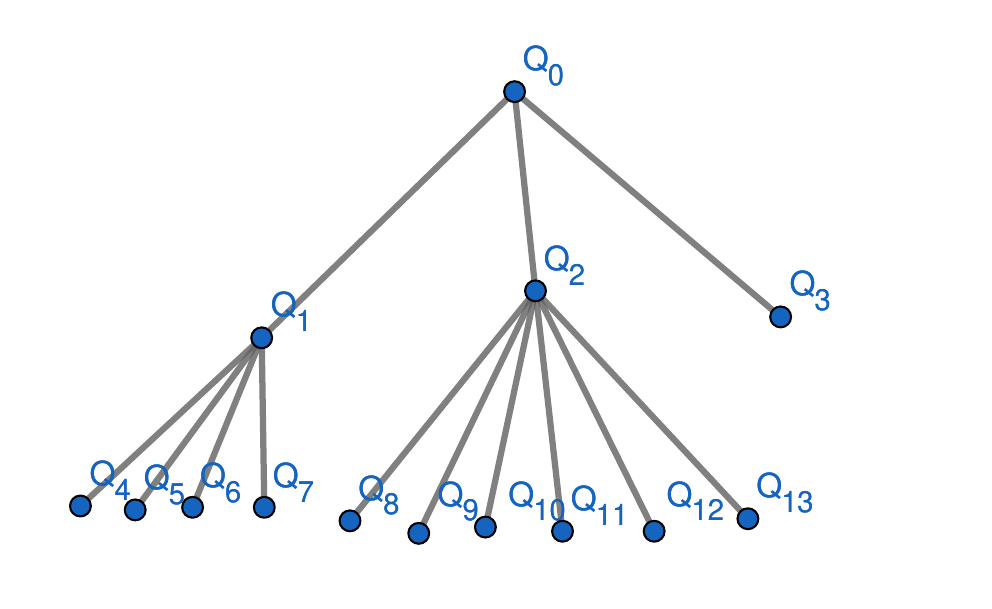}}
		\caption{Tree $\mathcal{T}_2$ for $M_0=3$, $M_1=4$, $M_2=6$.}\label{pic2}
	\end{figure}
	 We define the weight of the vertex $Q_i$ by the formula 
	\begin{equation}
		b(Q_i)=\!\!\!\!\underset{\text{ancestor of }Q_i}{\prod\limits_{Q_j\text{ is an}}}\!\!\!\! M_j^{-1}.
	\end{equation}

	We say that $Q_k$ belongs to the $n^{\text{th}}$ layer if it has exactly $n$ ancestors. We have bound $b(Q_k)\leqslant 2^{-n}$, if $Q_k$ belongs to the $n^{\text{th}}$ layer. The notation $n(Q_k)$ means the number of the layer of $Q_k$. 
	Clearly,
	\begin{equation}\label{probsum}
		\sum\limits_{n(Q_k)=n}b(Q_k)=1.
	\end{equation}
	
	Now let $Q_k$ (the vertices) be cubes in $\mathbb{R}^d$ and let $l(Q)$ the side length of cube $Q$.
	\begin{Def}
			The cube sequence $\{Q_0,Q_1,\cdots\}$  corresponds to the tree $\mathcal{T}$ if it satisfies requirements:
		\begin{enumerate}[1)]
			\item $Q_0=[0,1]^d$;
			
			\item if $Q_i$ is a parent of $Q_j$, then $Q_j\subset Q_i$;
			
			\item if $Q_j$ is a son of $Q_k$, then
			\begin{equation}\label{soot}
				l(Q_j)=r_k=\frac{M_k^{-\frac{p}{2d}}b(Q_k)^{\frac{p}{2d}}}{n(Q_k)+1}.
			\end{equation}
		\end{enumerate} 
	\end{Def}

	\begin{Def}
		Assume $\{Q_0,Q_1,\cdots\}$  corresponds to $\mathcal{T}$. The set $C_n$ is defined by the formula
		\begin{equation}
			C_n=\underset{n(Q_i)=n}{\bigcup}Q_i.
		\end{equation}
		We call the set $C=\cap_{n=1}^\infty C_n$ the Cantor-type set  corresponding to $\mathcal{T}$.
	\end{Def}
	Note that different songs of a cube may intersect.
	\begin{St}\label{H=0}
		Let $C$ be a Cantor-type set corresponding to $\mathcal{T}$. Then $\mathcal{H}_{\frac{2d}{p}}(C)=0$.
	\end{St}
	\begin{proof}
		The cubes of $n^\text{th}$ layer provide a covering of $C$. We use those coverings to estimate the Hausdorff measure of $C$:
		\begin{equation}
			\sum\limits_{n(Q_k)=n}\diam(Q_k)^{\frac{2d}{p}}\lesssim
			\sum\limits_{n(Q_k)=n}l(Q_k)^{\frac{2d}{p}}=
			\sum\limits_{n(Q_k)=n-1}M_kr_k^{\frac{2d}{p}}\overset{\text{\eqref{soot}}}{=}
			\sum\limits_{n(Q_k)=n-1}\frac{b(Q_k)}{n^{\frac{2d}{p}}}\ \overset{\eqref{probsum}}{=}\ 
			\frac{1}{n^{\frac{2d}{p}}}.
		\end{equation}
	\end{proof}
	There is a family of Cantor-type sets corresponding to the tree $\mathcal{T}$. To each Cantor-type set $C$, we will assign a probability measure $\mu$ such that $\supp(\mu)\subset C$. 
	
	Denote by $\lambda_Q$ the Lebesgue probability measure on the cube $Q$. Denote the measure $\mu_k$ by formula:
	\begin{equation}
		\mu_k=\sum\limits_{i=k+1}^{M_0+\dots+M_{k}}b(Q_i)\lambda_{Q_i}.
	\end{equation}
	These measures  satisfy recurrence relations:
	\begin{equation}
		\mu_0=\lambda_{[0,1]^d},
	\end{equation}
	\begin{equation}\label{nonrandomrel}
		\mu_k=\mu_{k-1}-b(Q_k)\lambda_{Q_k}+\frac{b(Q_k)}{M_k}\underset{\text{son of }Q_k}{\sum\limits_{Q_j \text{ is a}}}\lambda_{Q_j}.
	\end{equation}
	\subsection{Random construction}
	The behavior of the Fourier transform of a Cantor measure can be quite chaotic (see~\cite{Strichartz} for some examples). In this subsection we will use some randomization to choose a representative for which we can estimate $\hat{\mu}$.

	\begin{Def}
		Let $M\in\mathbb{N}$ and let $0<r<\frac{1}{2}$. Let $\mu_{M,r}$ be the random variable taking values in the set of probability measures:
		\begin{equation}
			\mu_{M,r}=\frac{1}{M}\sum\limits_{j=1}^{M}S_j\lambda_{[0,r]^d}.
		\end{equation}
		Here $\{S_j\}_{j=1}^{M}$ is a sequence of independent shifts that are uniformly distributed on the cube $[0,1-r]^d$.
	\end{Def}
	
	Let $f_r$ be a piecewise linear function (see Figure~\ref{pic1}) 
	
	\begin{equation}
		f_r(t)=
		\begin{cases}
			0, & t\in (-\infty,0]\cup[1,\infty),\\
			\frac{t}{r(1-r)}, & t\in [0,r],\\
			\frac{1}{1-r}, & t\in [r,1-r],\\
			\frac{1-t}{r(1-r)}, & t\in [1-r,1]. 
		\end{cases}
	\end{equation}
	\begin{figure}[h]
		\center{\includegraphics[scale=0.5]{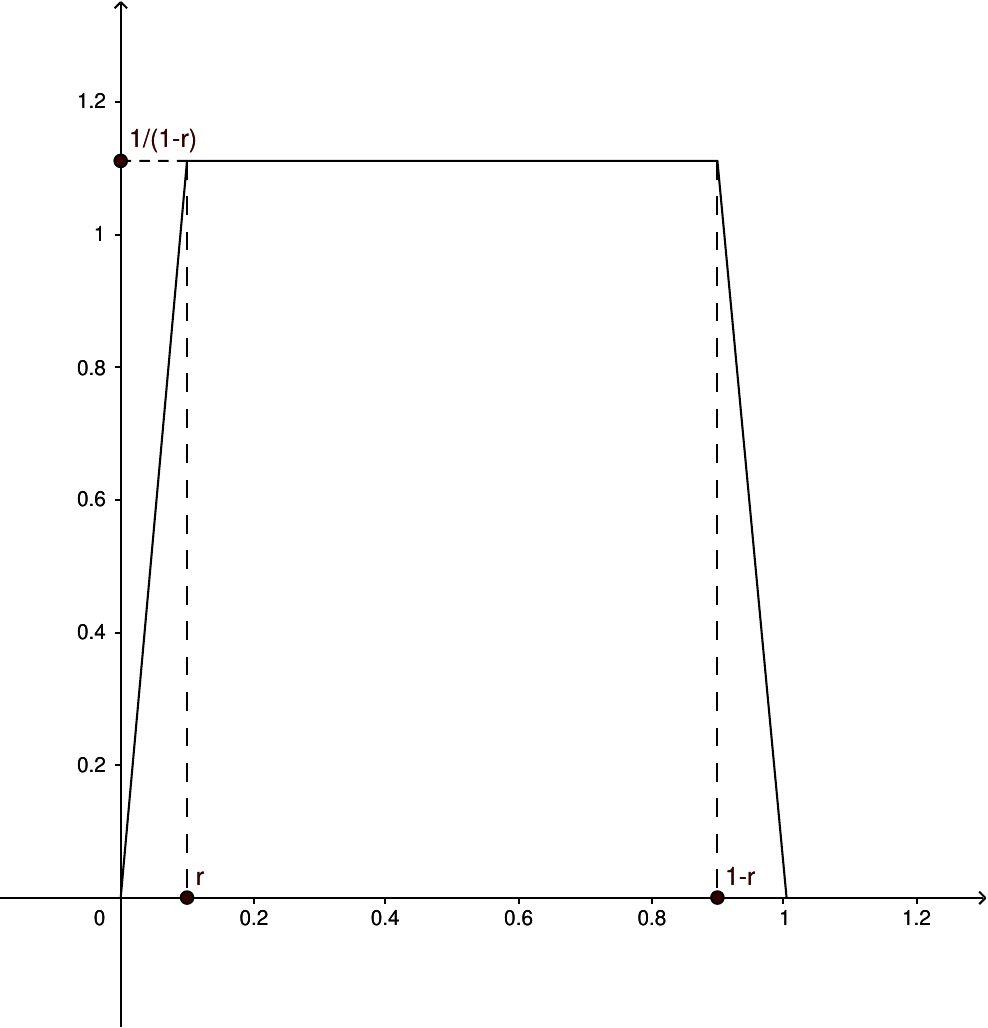}}
		\caption{Graph of $f_r$ for $r=0.1$}\label{pic1}
	\end{figure}
	
	Let $F_r(x)=\prod\limits_{i=1}^df_r(x_i)$. Then,
	\begin{equation}\label{mathex}
		\mathbb{E}\mu_{M,r}=\lambda_{[0,r]^d}*\lambda_{[0,1-r]^d}=F_r(x)dx
	\end{equation}
	 and  
	
	\begin{equation}\label{glav}
		\hat{\mu}_{M,r}(x)\overset{D}{=}\frac{1}{M}\sum_{j=1}^{M}e^{-2\pi i<\beta_{r,j},x>}\hat{\lambda}_0(xr).
	\end{equation}

	Here $\lambda_0$ is the Lebesgue measure on $[0,1]^d$ and $\{\beta_{r,j}\}_{j=1}^{M}$ is the sequence of independent vectors that are uniformly distributed on the cube $[0,1-r]^d$. The notation $\overset{D}{=}$ means equality of distributions.
	
	Let $\nu_{M,r}=\mu_{M,r}(\omega)$ be the value of $\mu_{M,r}$  at some point $\omega$ of the probability space to be chosen later.  By definitions, $\nu_{M,r}$ is a probability measure.
	
	We will use the measures $\nu_{M,r}$ to construct the sequences $\{Q_0,\dots,Q_{M_0+\dots+M_k}\}$ of cubes inductively. Set the cube $Q_0=[0,1]^d$. The measure $\nu_{M_0,r_0}$ corresponds to $M_0$  cubes. The sequence $\{Q_1,\dots,Q_{M_0}\}$ consists of those cubes.
	
	Assume we have constructed and the cubes $Q_0,Q_1,\dots,Q_{M_0+M_1+\dots+M_{k-1}}$. 
	Let $\nu_{M_k,\frac{r_k}{l(Q_k)},Q_k}$ be the image of $\nu_{M_k,\frac{r_k}{l(Q_k)}}$ under the homothety that maps $[0,1]^d$ into $Q_k$. The measure $\nu_{M_k,\frac{r_k}{l(Q_k)},Q_k}$ corresponds to $M_k$  cubes (see Figure~\ref{pic}). Add them to the end of the cube sequence. Thus, we have constructed the cubes $Q_0,Q_1,\dots,Q_{M_0+M_1+\dots+M_{k}}$.

	\begin{figure}[h]
		\center{\includegraphics[scale=0.5]{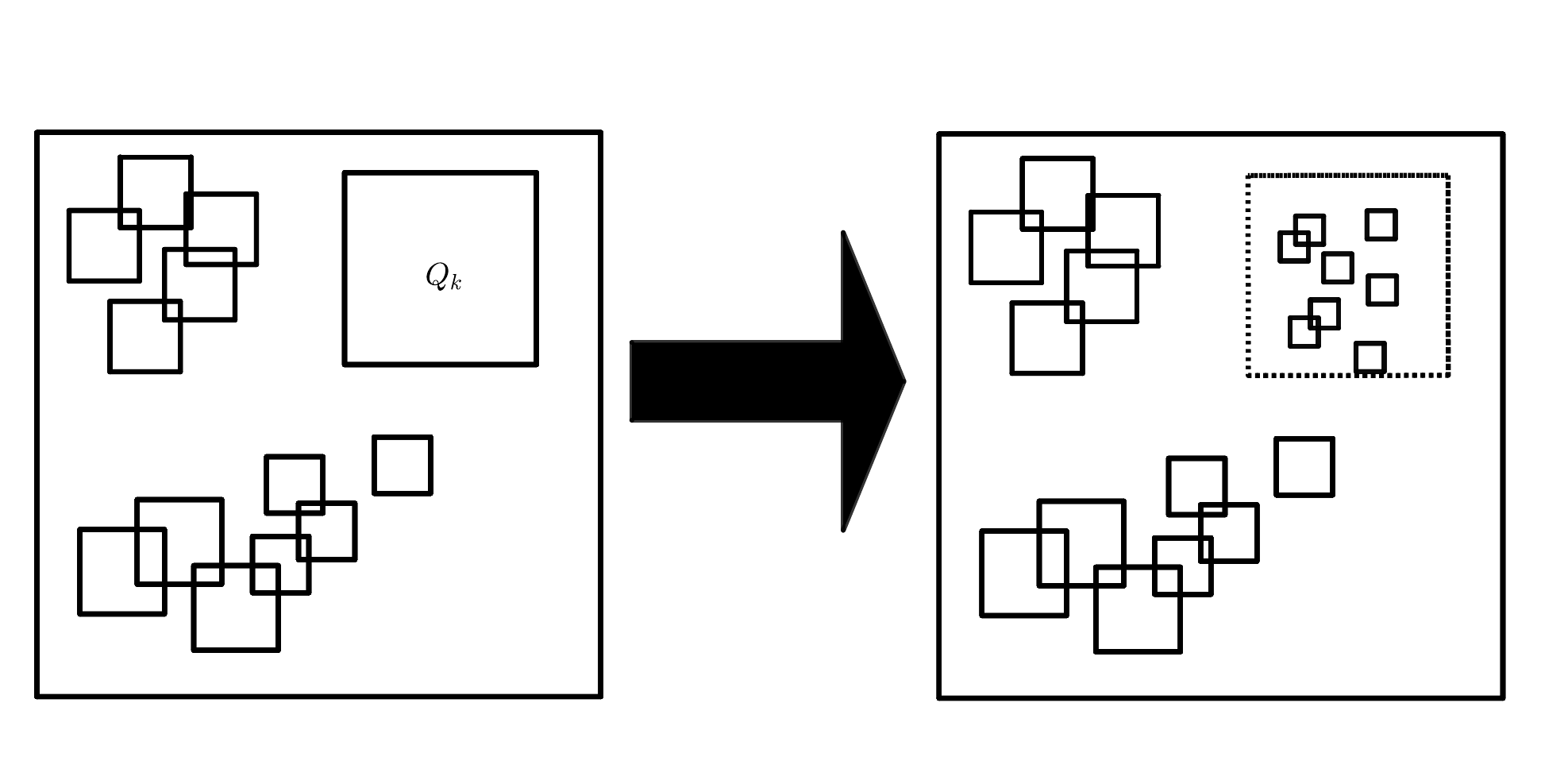}}
		\caption{Construction}\label{pic}
	\end{figure}
	
	\begin{Rem}
		The cube sequence $\{Q_0,Q_1,\cdots\}$ corresponds to a tree $\mathcal{T}$.
	\end{Rem}
	
	For this cube sequence the recurrence relations~\eqref{nonrandomrel} turn into
	
	\begin{equation}\label{formz}
		\mu_k=\mu_{k-1}-b(Q_k)\lambda_{Q_k}+b(Q_k)\nu_{M_k,\frac{r_k}{l(Q_k)},Q_k}.
	\end{equation}

	\section{Estimates}\label{ESC}

	\begin{Le}\label{Ep}
		For all $M\in \mathbb{N}$, $r<\frac{1}{2}$ and $p>1$ the inequality
		\begin{equation}\label{MP}
			\int\limits_{\mathbb{R}^d}|\mathbb{E}\hat{\mu}_{M,r}(x)|^pdx\lesssim 1
		\end{equation}
		is true.
	\end{Le}
	
	The constant depends only on $d$ and $p$.

	\begin{proof}
		By \eqref{glav}, we have
		\begin{equation}
			\mathbb{E}\hat{\mu}_{M,r}(x)=\mathbb{E}\frac{1}{M}\sum_{j=1}^{M}e^{-2\pi i<\beta_{r,j},x>}\hat{\lambda}_0(rx)=
			\hat{\lambda}_0(rx)\mathbb{E}e^{-2\pi i<\beta_{r,1},x>}.
		\end{equation}
		Since $\lambda_0$ is a probability measure, $|\hat{\lambda}_0|\leqslant 1$. With the notation $h(t)=\frac{1-e^{it}}{t}$, we have 
		\begin{multline}
			|\mathbb{E}\hat{\mu}_{M,r}(x)|=
			\left|\hat{\lambda}_0(rx)\mathbb{E}e^{-2\pi i<\beta_{r,1},x>}\right|\leqslant
			|\mathbb{E}e^{-2\pi i<\beta_{r,1},x>}|=\\
			\left|\frac{1}{(1-r)^d}\int\limits_{[0,1-r]^d}e^{-2\pi i<\beta_{r,1},x>}d\beta_{r,1}\right|=
			\left|\prod\limits_{j=1}^dh\left(2\pi (1-r)x_j\right)\right|\lesssim\\
			\prod\limits_{j=1}^{d}\min\left(1,\frac{1}{|x_j|}\right).
		\end{multline}
		The latter estimate implies \eqref{MP}.
	\end{proof}

	\begin{Le}\label{Np}
		For all $M\in \mathbb{N}$, $r<\frac{1}{2}$ and $p>1$ the inequality
		\begin{equation}
			\int\limits_{\mathbb{R}^d}\mathbb{E}|\hat{\mu}_{M,r}(x)-\mathbb{E}\hat{\mu}_{M,r}(x)|^pdx\lesssim M^{-\frac{p}{2}}r^{-d}
		\end{equation}
		is true.
	\end{Le} 
	
	\begin{proof}
		By~\eqref{glav} we have
		\begin{equation}
			\hat{\mu}_{M,r}(x)-\mathbb{E}\hat{\mu}_{M,r}(x)\overset{D}{=}\frac{1}{M}\sum_{j=1}^{M}\big(e^{-2\pi i<\beta_{r,j},x>}\hat{\lambda}_0(xr)-\mathbb{E}\hat{\mu}_{M,r}(x)\big).
		\end{equation}
		We use Lemma~\ref{mz} below: 
		\begin{multline}\label{glav1}
			\mathbb{E}|\hat{\mu}_{M,r}(x)-\mathbb{E}\hat{\mu}_{M,r}(x)|^p=
			\mathbb{E}\left|\frac{1}{M}\sum_{j=1}^{M}\left(e^{-2\pi i<\beta_{r,j},x>}\hat{\lambda}_0(xr)-\mathbb{E}\hat{\mu}_{M,r}(x)\right)\right|^p\leqslant\\
			C_pM^{-p}M^{\frac{p}{2}}\mathbb{E}\left|e^{-2\pi i<\beta_{r,1},x>}\hat{\lambda}_0(xr)-\mathbb{E}\hat{\mu}_{M,r}(x)\right|^p=
			C_pM^{-\frac{p}{2}}\mathbb{E}\left|e^{-2\pi i<\beta_{r,1},x>}\hat{\lambda}_0(xr)-\mathbb{E}\hat{\mu}_{M,r}(x)\right|^p\leqslant\\
			C_pM^{-\frac{p}{2}}2^{p-1}\left(\left|\hat{\lambda}_0(xr)\right|^p+\left|\mathbb{E}\hat{\mu}_{M,r}(x)\right|^p\right).
		\end{multline}
		We integrate \eqref{glav1} to complete the proof:
		\begin{equation}
			\int\limits_{\mathbb{R}^d}\mathbb{E}|\hat{\mu}_{M,r}(x)-\mathbb{E}\hat{\mu}_{M,r}(x)|^pdx\leqslant 
			\int\limits_{\mathbb{R}^d}
			C_pM^{-\frac{p}{2}}2^{p-1}\left(\left|\hat{\lambda_0}(xr)\right|^p+\left|\mathbb{E}\hat{\mu}_{M,r}(x)\right|^p\right)dx\lesssim M^{-\frac{p}{2}}r^{-d}.
		\end{equation}
		In final step, we have used Lemma~\ref{Ep}.
	\end{proof}

	Fix $p_1$ such that $p_1>p$.
	
	\begin{Cor}\label{PrE}
		For all $M\in\mathbb{N}$ and $0<r<\frac{1}{2}$ we can choose measures $\nu_{M,r}=\mu_{M,r}(\omega)$ such that
		
		\begin{equation}
			\int\limits_{\mathbb{R}^d}|\hat{\nu}_{M,r}(x)-\mathbb{E}\hat{\mu}_{M,r}(x)|^{p_1}dx\lesssim M^{-\frac{p_1}{2}}r^{-d},
		\end{equation}
		\begin{equation}
			\int\limits_{\mathbb{R}^d}|\hat{\nu}_{M,r}(x)-\mathbb{E}\hat{\mu}_{M,r}(x)|^{p}dx\lesssim M^{-\frac{p}{2}}r^{-d}.
		\end{equation}
	\end{Cor}
	
	\begin{Le}\label{ooo}
		Let $\lambda_0$ be the Lebesgue measure on $[0,1]^d$. Then for $p>2$ the inequality
		\begin{equation}
			\|\hat{\lambda}_0-\mathbb{E}\hat{\mu}_{M,r}\|_{L_{p}}\lesssim r^{\frac{1}{p'}}
		\end{equation}
		is true.
	\end{Le}
	\begin{proof}
		One may see from~\eqref{mathex} that $\|\lambda_0-\mathbb{E}\mu_{M,r}\|_{L_{p'}}\lesssim r^{\frac{1}{p'}}$. The Hausdorff--Young inequality finishes the proof.
	\end{proof}
	
	\begin{Cor}\label{cor1}
		We have the inequalities
		\begin{equation}
			\|\hat{\nu}_{M,r}(x)-\hat{\lambda}_0(x)\|^p_{L_{p}}\lesssim M^{-\frac{p}{2}}r^{-d} + r^{\frac{p}{p'}},
		\end{equation}
		\begin{equation}
			\|\hat{\nu}_{M,r}(x)-\hat{\lambda}_0(x)\|^{p_1}_{L_{p_1}}\lesssim M^{-\frac{p_1}{2}}r^{-d}+r^{\frac{p_1}{p_1'}}.
		\end{equation}
		Here $\nu_{M,r}$ is defined by the choice of $\omega$ in Corollary~\ref{PrE}.
	\end{Cor}

	\begin{Le}\label{Ras}
		Let $M_0,M_1,M_2,\dots,M_{k-1}$ be fixed and let $M_k$ tends to infinity. Then, the following inequalities
		\begin{equation}
			\varlimsup\limits_{M_k\rightarrow\infty}\|\hat{\mu}_k-\hat{\mu}_{k-1}\|^p_{L_{p}}\lesssim b(Q_k)^{\frac{p}{2}}(n(Q_k)+1)^d,
		\end{equation}
		\begin{equation}
			\lim\limits_{M_k\rightarrow\infty}\|\hat{\mu}_k-\hat{\mu}_{k-1}\|_{L_{p_1}}=0
		\end{equation}
		are true.
	\end{Le}
	\begin{proof}
		We can write estimates
	\begin{multline}
		\|\hat{\mu}_k-\hat{\mu}_{k-1}\|^p_{L_{p}}\overset{\text{\eqref{formz}}}{=}
		\|b(Q_k)(\hat{\nu}_{M_k,\frac{r_k}{l(Q_k)},Q_k}-\hat{\lambda}_{Q_k})\|^p_{L_{p}}=
		b(Q_k)^p\|\hat{\nu}_{M_k,\frac{r_k}{l(Q_k)},Q_k}-\hat{\lambda}_{Q_k}\|^p_{L_{p}}=\\
		b(Q_k)^pl(Q_k)^{-d}\|\hat{\nu}_{M_k,\frac{r_k}{l(Q_k)}}-\hat{\lambda}_0\|^p_{L_{p}}\overset{\text{Cor~\ref{cor1}}}{\lesssim}
		b(Q_k)^pl(Q_k)^{-d}\left(M_k^{-\frac{p}{2}}\left(\frac{r_k}{l(Q_k)}\right)^{-d} + \left(\frac{r_k}{l(Q_k)}\right)^{\frac{p}{p'}}\right)\overset{\text{\eqref{soot}}}{=}\\
		b(Q_k)^{\frac{p}{2}}(n(Q_k)+1)^d+b(Q_k)^{p}l(Q_k)^{-d-\frac{p}{p'}}r_k^{\frac{p}{p'}}\overset{M_k\rightarrow\infty}{\rightarrow} b(Q_k)^{\frac{p}{2}}(n(Q_k)+1)^d
	\end{multline}
	and in the same way we have
	\begin{equation}
		\|\hat{\mu}_k-\hat{\mu}_{k-1}\|^{p_1}_{L_{p_1}}\lesssim
		b(Q_k)^{p_1-\frac{p}{2}}M_k^{-\frac{p_1-p}{2}}(n(Q_k)+1)^d+
		b(Q_k)^{p_1}l(Q_k)^{-d}\left(\frac{r_k}{l(Q_k)}\right)^{\frac{p_1}{p_1'}}
		\overset{M_k\rightarrow\infty}{\rightarrow}0.
	\end{equation}
	Here  $b(Q_k)$, $n(Q_k)$ and $l(Q_k)$ are fixed and $r_k\rightarrow0$ while $M_k\rightarrow\infty$ (see equation~\eqref{soot}).
	\end{proof}

	\begin{Cor}
		There exists a constant $C>0$ such that inequality
		\begin{equation}
			\|\hat{\mu}_{k}\|^p_{L_{p}}\leqslant \|\hat{\mu}_{k-1}\|^p_{L_{p}}+Cb(Q_k)^{\frac{p}{2}}(n(Q_k)+1)^d
		\end{equation}
		is true, provided $M_k$ is  sufficiently large.
	\end{Cor}
	This corollary is a combination of the previous lemma and Lemma~\ref{p+p} below.

	Let $\{M_0,M_2,\dots\}$ be rapidly growing infinite sequence. Let $\mu_{\infty}$ be a weak* limit of $\mu_{k}$. The rapid grouth of  $\mathcal{M}_\infty$ provides us with the following inequality
	\begin{multline}
		\|\hat{\mu}_{\mathcal{M}_k}\|^p_{L_p}\leqslant \|\hat{\lambda}_0\|^p+C\sum\limits_{j=0}^{k}b(Q_j)^{\frac{p}{2}}n(Q_j)^d\lesssim
		\sum\limits_{j=0}^{\infty}b(Q_j)^{\frac{p}{2}}(n(Q_j)+1)^d=
		\sum\limits_{n=0}^{\infty}\sum\limits_{n(Q_j)=n}b(Q_j)^{\frac{p}{2}}(n+1)^d\leqslant\\
		\sum\limits_{n=0}^{\infty}\sum\limits_{n(Q_j)=n}b(Q_j)2^{-n(\frac{p}{2}-1)}(n+1)^d\ 
		\overset{\eqref{probsum}}{=}\ 
		\sum\limits_{n=0}^{\infty} 2^{-n(\frac{p}{2}-1)}(n+1)^d<\infty.
	\end{multline}
	Thus $\hat{\mu}_{\infty}\in L_p(\mathbb{R}^d)$. This formula and Proposition~\ref{H=0} completes the proof of Theorem~\ref{res}.

	\section{Auxiliary lemmas}\label{AL}
	
	\begin{Le}\label{mz}
		Let $\{X_j\}_{j=1}^M$ be i.i.d. random variables with zero mean. Then, for $p>2$ the inequality
		\begin{equation}
			\mathbb{E}\left|\sum_{j=1}^{M}X_j\right|^p\leqslant C_pM^{\frac{p}{2}}\mathbb{E}|X_1|^p
		\end{equation}
		is true. Here $C_p$ is a constant that does not depend on the $X_j$.
	\end{Le}
	\begin{proof}
		First, we use the Marcinkiewicz--Zygmund inequality (see~\cite[Chapter~10.3, Theorem 2]{ChowTeicher})
		\begin{equation}
			\mathbb{E}\left|\sum_{j=1}^{M}X_j\right|^p\leqslant C_p\mathbb{E}\left(\sum_{j=1}^{M}|X_j|^2\right)^\frac{p}{2}.
		\end{equation}
		Second, we use H\"older's inequality to obtain:
		\begin{equation}
			\mathbb{E}\left(\sum_{j=1}^{M}|X_j|^2\right)^\frac{p}{2}\leqslant\mathbb{E} M^{\frac{p}{2}-1}\sum\limits_{j=1}^{M}|X_j|^p=M^{\frac{p}{2}}\mathbb{E}|X_1|^p.
		\end{equation}
	\end{proof}

	\begin{Le}\label{p+p}
		Let $f\in L_{p}$ be a function and let $g_j\in L_{p}$ be a sequence of functions. Assume that $\varlimsup\limits_{j\rightarrow\infty}\|g_j\|_{L_{p}}\leqslant A$ and for some $p_1>p$ we have $\lim\limits_{j\rightarrow\infty}\|g_j\|_{L_{p_1}}=0$. Then
		\begin{equation}
			\varlimsup\limits_{j\rightarrow\infty} \|f+g_j\|^p_{L_{p}}\leqslant \|f\|^p_{L_{p}}+A^p.
		\end{equation} 
	\end{Le}
	
	\begin{proof}
		For a function $h$ and $\varepsilon>0$, let 
		\begin{equation}
			h_{\geqslant \varepsilon}=\chi_{\{|h|\geqslant \varepsilon\}}h,
		\end{equation}
		\begin{equation}
			h_{< \varepsilon}=\chi_{\{|h|< \varepsilon\}}h.
		\end{equation}
		
		Clearly, $h=h_{\geqslant \varepsilon}+h_{< \varepsilon}$. Pick some $\delta>0$. Then
		\begin{multline}\label{texv}
			\|f+g_j\|_{L_p}\leqslant \|f_{\geqslant\varepsilon}+{g_j}_{<\delta}\chi_{\{f< \varepsilon\}}\|_{L_p}+\|{g_j}_{\geqslant\delta}\|_{L_p}+\|f_{<\varepsilon}\|_{L_p}+\|{g_j}_{<\delta}\chi_{\{f\geqslant \varepsilon\}}\|_{L_p}=\\
			(\|f_{\geqslant\varepsilon}\|^p_{L_p}+\|{g_j}_{<\delta}\chi_{\{f< \varepsilon\}}\|^p_{L_p})^{\frac{1}{p}}+\|{g_j}_{\geqslant\delta}\|_{L_p}+\|f_{<\varepsilon}\|_{L_p}+\|{g_j}_{<\delta}\chi_{\{f\geqslant \varepsilon\}}\|_{L_p}\leqslant\\
			(\|f\|^p_{L_p}+\|g_j\|^p_{L_p})^{\frac{1}{p}}+\|{g_j}_{\geqslant\delta}\|_{L_p}+\|f_{<\varepsilon}\|_{L_p}+\delta\|\chi_{\{f\geqslant \varepsilon\}}\|_{L_p}.
		\end{multline}

		The equality $\lim\limits_{j\rightarrow\infty}\|g_j\|_{L_{p_1}}=0$ implies   $\lim\limits_{j\rightarrow\infty}\|{g_j}_{\geqslant\delta}\|_{L_p}=0$. Therefore, \eqref{texv} yields
		\begin{multline}
			\varlimsup\limits_{j\rightarrow\infty} \|f+g_j\|_{L_p}\leqslant (\|f\|^p_{L_p}+A^p)^{\frac{1}{p}}+\|f_{<\varepsilon}\|_{L_p}+\delta\|\chi_{\{f\geqslant \varepsilon\}}\|_{L_p}\overset{\delta\rightarrow0}{\rightarrow} (\|f\|^p_{L_p}+A^p)^{\frac{1}{p}}+\|f_{<\varepsilon}\|_{L_p}\overset{\varepsilon\rightarrow0}{\rightarrow}\\
			(\|f\|^p_{L_p}+A^p)^{\frac{1}{p}}.
		\end{multline}

	\end{proof}

	\
	
	\
	
	\
	
	{\small
		
		St. Petersburg State University Department of Leonhard Euler International Mathematical Institute;\\
		e-mail:dobronravov1999@mail.ru\\
		\bigskip
		
	}
	
\end{document}